\documentclass[a4paper,11pt]{article}

\usepackage[T1]{fontenc}
\usepackage[utf8]{inputenc}
\usepackage{latexsym}
\usepackage{amsmath,amsthm,amssymb,amsfonts}
\usepackage{eucal}
\usepackage[english]{babel}
\usepackage{tikz}
\usepackage{tkz-graph}
\usepackage{graphics}

\newtheorem{theorem}{Theorem} 
\newtheorem{proposition}[theorem]{Proposition}
\newtheorem{lemma}[theorem]{Lemma}
\newtheorem{corollary}[theorem]{Corollary}

\newtheorem*{openproblem*}{Open Problem}

\theoremstyle{definition}

\newtheorem*{remark*}{Remark}

\newtheorem*{example*}{Example}

\newcommand{\ZZ}{\mathbb{Z}}

\newcommand{\EGZ}{\mathrm{EGZ}}
\newcommand{\pt}{\hspace{1.5pt}\begin{picture}(-1,1)(-1,-3)\circle*{2.5}\end{picture}\hspace{3pt}}
\newcommand{\vpt}{\hspace{1.5pt}\begin{picture}(-1,1)(-1,-3)\circle{2.5}\end{picture}\hspace{3pt}}
\newcommand{\ptS}{S^{\pt}}
\newcommand{\vptS}{S^{\vpt}}
\newcommand{\ptT}{T^{\pt}}
\newcommand{\Sym}{\mathrm{Sym}}

\begin{document}

\title{Counting tournament score sequences}

\author{Anders Claesson \and
  Mark Dukes \and
  Atli Fannar Franklín \and
  Sigurður Örn Stefánsson
}
\date{12 January 2022}
\maketitle

\begin{abstract}
  The score sequence of a tournament is the sequence of the out-degrees
  of its vertices arranged in nondecreasing order. The problem of
  counting score sequences of a tournament with $n$ vertices is more
  than 100 years old (MacMahon 1920). In 2013 Hanna conjectured a
  surprising and elegant recursion for these numbers. We settle this
  conjecture in the affirmative by showing that it is a corollary to our
  main theorem, which is a factorization of the generating function for
  score sequences with a distinguished index. We also derive a
  closed formula and a quadratic time algorithm for counting score sequences.
\end{abstract}
\thispagestyle{empty}

\section{Introduction}\label{section-intro}

In 1953 Landau~\cite{Landau1953} used oriented complete graphs---also
called \emph{tournaments}---to model pecking orders.  If the vertices of
the complete graph represent players (rather than chickens), then the
initial vertex of a directed edge signifies the winner of a game between
the two end-point players. The number of wins of a player is equal to
the number of outgoing edges from that vertex. A \emph{score sequence} is
a sequence of these number of wins given in a nondecreasing order. For
instance, with 3 players there are two possible score sequences, namely
$(0, 1, 2)$ and $(1, 1, 1)$.  Note that non-isomorphic tournaments may
give rise to the same score sequence. With $5$ players there are, up to
isomorphism, 12 tournaments but only 9 score sequences. To be even more
specific, here are two non-isomorphic\footnote{The longest directed path
  between any pair of vertices in the left-hand graph is 3, while in the
  right-hand graph there is a directed path of length $4$ from $c$ to $b$.}
tournaments:
\[
\begin{tikzpicture}[rotate=90, scale=0.95]
  \GraphInit[vstyle=Classic]
  \tikzset{VertexStyle/.append style={inner sep=2.6pt,outer sep=0.3pt,minimum size=2pt,fill=white}}
  \tikzset{EdgeStyle/.append style={semithick}}
  \Vertices[Math,Lpos=90,unit=1.1]{circle}{a,b,c,d,e}
  \Edge[lw=0.1cm,style={post}](a)(c)
  \Edge[lw=0.1cm,style={post}](a)(d)
  \Edge[lw=0.1cm,style={post}](b)(a)
  \Edge[lw=0.1cm,style={post}](b)(c)
  \Edge[lw=0.1cm,style={post}](b)(d)
  \Edge[lw=0.1cm,style={post}](c)(d)
  \Edge[lw=0.1cm,style={post}](d)(e)
  \Edge[lw=0.1cm,style={post}](e)(a)
  \Edge[lw=0.1cm,style={post}](e)(b)
  \Edge[lw=0.1cm,style={post}](e)(c)
\end{tikzpicture}\qquad
\begin{tikzpicture}[rotate=90, scale=0.95]
  \GraphInit[vstyle=Classic]
  \tikzset{VertexStyle/.append style={inner sep=2.6pt,outer sep=0.3pt,minimum size=2pt,fill=white}}
  \tikzset{EdgeStyle/.append style={semithick}}
  \Vertices[Math,Lpos=90,unit=1.1]{circle}{a,b,c,d,e}
  \Edge[lw=0.1cm,style={post}](a)(b)
  \Edge[lw=0.1cm,style={post}](a)(c)
  \Edge[lw=0.1cm,style={post}](a)(d)
  \Edge[lw=0.1cm,style={post}](b)(c)
  \Edge[lw=0.1cm,style={post}](b)(d)
  \Edge[lw=0.1cm,style={post}](b)(e)
  \Edge[lw=0.1cm,style={post}](c)(d)
  \Edge[lw=0.1cm,style={post}](d)(e)
  \Edge[lw=0.1cm,style={post}](e)(a)
  \Edge[lw=0.1cm,style={post}](e)(c)
\end{tikzpicture}
\]
The score sequence associated with both is $(1,1,2,3,3)$. The following
characterization of score sequences is known as Landau's theorem.
\begin{theorem}[Landau~\cite{Landau1953}]\label{landaus-thm}
  A sequence of integers $s=(s_0,\ldots,s_{n-1})$ is a score sequence if
  and only if
  \begin{enumerate}
  \item[\textup{(1)}] $\refstepcounter{equation} \label{landaus-thm-i} 0 \leq s_0 \leq s_1 \leq \cdots \leq s_{n-1} \leq n-1 $,
  \item[\textup{(2)}] $\refstepcounter{equation} \label{landaus-thm-ii} s_0+\dots+s_{k-1} \geq {k \choose 2}$ for $1\leq k < n$, and
  \item[\textup{(3)}] $\refstepcounter{equation} \label{landaus-thm-iii} s_0+\dots+s_{n-1}={n \choose 2}$.
  \end{enumerate}
\end{theorem}

Let $S_n$ be the set of score sequences of length $n$. There is no known
closed formula for the associated cardinalities (A000571 in the OEIS~\cite{OEIS})
\[
  \bigl(|S_n|\bigr)_{n\geq 0} =
  (1,1,1,2,4,9,22,59,167,490,1486,4639,14805,\dots)
\]
or their generating function.

It should be noted that Landau was not the first person to study
score sequences, or attempt to count them. MacMahon~\cite{MacMahon1920}
used symmetric functions and hand calculations to determine $|S_n|$ for
$n\leq 9$ in 1920. Building on Landau's work, Narayana and
Bent~\cite{NarayanaBent1964}, in 1964, derived a multivariate recursive
formula for determining $|S_n|$. They used it to give a table for
$n\leq 36$. In 1968 Riordan~\cite{Riordan1968} gave a simpler and more
efficient recursion, but unfortunately it turned out to be incorrect~\cite{Riordan1969}.

Let $[a,b]$ denote the interval of integers $\{a, a+1,\dots,b\}$.  We
may view a score sequence $s\in S_n$ as an endofunction
$s:[0,n-1]\to [0,n-1]$.  We now introduce the notion of a \emph{pointed
  score sequence}. Define $\ptS_n$ as the Cartesian product
$\ptS_n=S_n\times [0,n-1]$. We call the members of $\ptS_n$
\emph{pointed score sequences}; e.g.\ there are 6 pointed score
sequences in $\ptS_3$:
\begin{align*}
  & ((0,1,2),0),\, ((0,1,2),1),\, ((0,1,2),2),\, \\
  & ((1,1,1),0),\, ((1,1,1),1),\, ((1,1,1),2).
\end{align*}
Let $(s,i)\in \ptS_n$. Depending on the
context, the element $i$ will be interpreted as a position
(element in the domain) or a  value (element in the
codomain) of $s$. If $i$ is a value, then the cardinality
of the fiber $s^{-1}(i)$ is the number of times $i$ occurs in $s$; this
number may be zero. Let
\[\ptS_n(t) = \sum_{(s,i)\in \ptS_n} t^{|s^{-1}(i)|}
\]
be the polynomial recording the distribution of the statistic
$(s,i)\mapsto |s^{-1}(i)|$ on $\ptS_n$. As an example,
$\ptS_3(t) = 2+3t+t^3$. Let
\[
\ptS(x,t) = \sum_{n\geq 1} \ptS_n(t) x^n.
\]

To present the bijection that is the main result of this paper,
we will first introduce a particular type of multiset that is an
essential ingredient in our deconstruction of a pointed score sequence.
At first glance it is not obvious what the relevance of these multisets
to score sequences is.

We define $\EGZ_n$ as the set of multisets of size $n$ with elements in
the cyclic group $\ZZ_n$ whose sum is $\binom{n}{2}$ modulo $n$. To
understand what the elements of $\EGZ_n$ look like it may be helpful to
note that $\binom{n}{2}$, as an element of $\ZZ_n$, is $0$ if $n$ is odd and $n/2$ if
$n$ is even. For instance, $\EGZ_3$ consists of the $4$
multisets
\begin{equation*}
  \{0,0,0\},\; \{0,1,2\},\; \{1,1,1\},\;\,\text{and}\,\;\{2,2,2\}.
\end{equation*}

The notation $\EGZ_n$ refers to the Erd\H{o}s-Ginzburg-Ziv
Theorem~\cite{EGZ1961}, which is stated below. Following it we give a
proposition motivating this terminology; its proof gives a simple
one-to-one correspondence between $\EGZ_n$ and the sets considered by
Erd\H{o}s, Ginzburg, and Ziv.

\begin{theorem}[Erd\H{o}s, Ginzburg, and Ziv~\cite{EGZ1961}]\label{EGZ-Theorem}
  Each set of $2n-1$ integers contains some subset of $n$ elements the
  sum of which is a multiple of $n$.
\end{theorem}

\begin{proposition}
  There is a one-to-one correspondence between $\EGZ_n$ and $n$-element
  subsets of $[1,2n-1]$ whose sum is a multiple of $n$.
\end{proposition}

\begin{proof}
  Let $A=\{a_1,\dots,a_n\}$ be a subset of $[1,2n-1]$ such that
  $a_1+\cdots+a_n$ is divisible by $n$.  Without loss of generality we can
  further assume that $a_1<a_2<\cdots <a_n$. Let $b_i=a_i-i$. We claim
  that $A\mapsto \{b_1,\ldots, b_n\}$ is a bijection onto $\EGZ_n$.
  Clearly, $i\leq a_i\leq n+i-1$ and hence $0\leq b_i\leq n-1$. In this manner we
  can consider each $b_i$ as an element of $\ZZ_n$. The sum
  $a_1+\cdots+a_n$ is divisible by $n$, by assumption, and hence
  \[b_1 + \dots + b_n
    \equiv a_1 + \dots + a_n - 1 - 2 - \dots - n
    \equiv - \binom{n + 1}{2} \pmod{n}.
  \]
  That this is congruent to $\binom{n}{2}$ modulo $n$ follows from
  $\binom{n}{2}+\binom{n + 1}{2}=n^2$.  Conversely, given a multiset
  $B=\{b_1,\dots,b_n\}$, with $b_1\leq\cdots\leq b_n$, in $\EGZ_n$ we
  define the set $A=\{a_1,\dots,a_n\}$ by $a_i=b_i+i$. One can verify
  that $A$ is a subset of $[1,2n-1]$ whose sum is divisible by $n$, and
  that $A\mapsto B$ is the inverse of the map above, but we omit the
  details.
\end{proof}

The sequence of cardinalities
\[\bigl(|\EGZ_n|)_{n\geq 1} = (1, 1, 4, 9, 26, 76, 246, 809, 2704, 9226, 32066, \dots)
\]
is entry A145855
in the OEIS~\cite{OEIS}. As recorded in that OEIS entry, Jovovi\'c
conjectured and Alekseyev~\cite{Alekseyev2008} proved in 2008 that
\begin{equation}\label{EGZ-cardinality}
  |\EGZ_n| = \frac{1}{2n}\sum_{d | n}(-1)^{n-d}\varphi(n/d)\binom{2d}{d},
\end{equation}
where the sum runs over all positive divisors of $n$ and $\varphi$ is
Euler's totient function.  A generalization of this result was given by
Chern~\cite{Chern2019} in 2019.

The zeros in a multiset $M \in \EGZ_n$ play a prominent role in our
construction. We now introduce a generating function to record their
number.
For a multiset $M\in\EGZ_n$ let $|M|_i$ be the number of occurrences of $i$ in
$M$. Furthermore, let
\[\EGZ_n(t) = \sum_{M\in\EGZ_n} t^{|M|_0}
\]
be the polynomial recording the distribution of zeros in multisets
belonging to $\EGZ_n$. For instance,
$\EGZ_3(t)=2+t+t^3$ (looking at the distribution of $1$s or
$2$s in $\EGZ_3$ would result in the same polynomial).
%
%
Define the generating functions
\begin{equation*}
  \EGZ(x,t) = \sum_{n\geq 1}\EGZ_n(t) x^n \quad\text{and}\quad
  S(x) = \sum_{n\geq 0}|S_n| x^n.
\end{equation*}
Our main result (Theorem~\ref{main-result}) is a
factorization of the generating function for pointed score sequences:
\begin{equation}\label{main-eq}
  \ptS(x,t) = \EGZ(x,t)S(x).
\end{equation}
Let $(s,i)\in\ptS_n$. Viewing $i$ is an element of the codomain of $s$ we find
that $\ptS(x,0)$ consists of terms stemming from pairs $(s,i)$ such that
$s^{-1}(i)$ is empty; i.e.\ $i$ is outside the image of $s$. Thus,
$\ptS(x,1)-\ptS(x,0)$ counts pairs $(s,i)$ for which $i$ is in the image
of $s$. Let
\begin{align*}
  \vptS_n
  &= \{(s,i)\in\ptS_n: i\in\textrm{Im}(s)\} \\
  &= \{(s,i)\in\ptS_n: \text{$i=s_j$ for some $j\in[n]$}\}
\end{align*}
and let $\vptS(x)=\ptS(x,1)-\ptS(x,0)$ be the corresponding generating
function. For instance, $\vptS_3$ consists of the 4 elements
$((0,1,2),0)$, $((0,1,2),1)$, $((0,1,2),2)$, and $((1,1,1),1)$.  We will
show (in Corollary~\ref{catalan-conundrum}) that
\[
  \vptS(x) = xC(x)S(x),
\]
where $C(x)=(1-\sqrt{1-4x})/(2x)$ is the generating function for the
Catalan numbers $C_n=\binom{2n}{n}/(1+n)$. This striking occurrence of
the Catalan numbers was in fact the original inspiration for our work.
It was in the summer of 2019 that we experimented with score sequences
and conjectured the identity. Despite ample attempts we were for the
longest time unable to prove it.

By setting $t=1$ in Equation~\ref{main-eq} and noting that
$\ptS(x,1)=xS'(x)$ it follows that
\begin{equation}\label{main-eq-1}
  xS'(x)=\EGZ(x,1)S(x),
\end{equation}
a fact conjectured by Paul D.\ Hanna as recorded in the OEIS entry
A000571 in 2013. Equation~\ref{main-eq-1} may alternatively be written
$\bigl(\log S(x)\bigr)'=\EGZ(x,1)/x$ and so
\[
  S(x) \,=\, \exp\Biggl(\,\sum_{n\geq 1} \frac{|\EGZ_n|}{n} x^n\Biggr),
\]
which arguably is the most elegant way of expressing the relation
between $|S_n|$ and $|\EGZ_n|$. The
most efficient way of computing the numbers $|S_n|$ is, however, to use
the recursion underlying Equation~\ref{main-eq-1}. Namely,
 $|S_0|=1$ and, for $n\geq 1$,
\[
  |S_n| = \frac{1}{n}\sum_{k=1}^{n}|S_{n-k}||\EGZ_k|.
\]
See Corollary~\ref{recursion} and the discussion following it.

\section{The main theorem and its bijection}\label{section-main-thm}

Let the generating functions $\ptS(x,t)$, $\EGZ(x,t)$ and $S(x)$
be defined as in Section~\ref{section-intro}.

\begin{theorem}\label{main-result} We have
  \[\ptS(x,t) = \EGZ(x,t)S(x).
  \]
\end{theorem}

We shall give a combinatorial proof of Theorem~\ref{main-result} using a
bijection
\[
  \Phi: \ptS_n \to \bigcup_{k=1}^n \EGZ_k\times S_{n-k}
\]
that maps a pointed score sequence to a pair consisting of a multiset
and a score sequence.  A property of this bijection is that, for
$(M,v)=\Phi(s,i)$, the number of occurrences of $i$ in $s$ is equal to
the multiplicity of zero in $M$. Before defining $\Phi$ we need to
introduce several necessary concepts.

A nonempty directed graph is said to be \emph{strongly connected} if there is a
directed path between each pair of vertices of the graph. Note that we do not
consider the empty graph to be strongly connected. A \emph{strong
score sequence} is one which stems from a strongly connected
tournament. Equivalently (see Harary and
Moser~\cite[Theorem~9]{HararyMoser1966}), $s = (s_0,\dots,s_{n-1})$, with $n\geq 1$, is a
strong score sequence if the inequality \eqref{landaus-thm-ii} of
Theorem~\ref{landaus-thm} is always strict; that is,
$s_0+\dots+s_{k-1}>{k \choose 2}$ for $1\leq k < n$.  Let us define the
\emph{direct sum} of two score sequences $u\in S_k$ and $v\in S_{\ell}$
by $u\oplus v = uv'$, where $v'$ is obtained from $v$ by adding $k$ to
each of its letters and juxtaposition indicates concatenation.  For
instance, $(0)\oplus (0)\oplus (1,1,1)=(0,1,3,3,3)$.  If $U$ and $V$ are
tournaments having score sequences $u$ and $v$, one may view the direct
sum $u\oplus v$ as the score sequence of the tournament where arrows are
placed between the vertices of $U$ and $V$ such that they all point
towards $U$:
\begin{center}
  \includegraphics[height=2.3em]{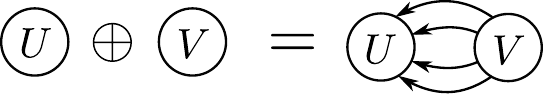}
\end{center}
This may easily be seen to be
independent of the choice of tournaments.

\begin{lemma}\label{direct-sum-lemma}
  Let $s\in S_n$. If $s_0+\cdots +s_{k-1} = \binom{k}{2}$ for some
  $k<n$, then $u=(s_0,\dots,s_{k-1})$ and $v=(s_k-k, \dots, s_{n-1}-k)$
  are both score sequences, and $s=u\oplus v$.
\end{lemma}

\begin{proof}
  That $u$ is a score sequence is clear from Landau's theorem. By
  definition we have $s=u\oplus v$, so it only remains to show that $v$
  is a score sequence and we will use Landau's theorem to do so, proving
  each of the three parts separately:
  \begin{enumerate}
    \item[\eqref{landaus-thm-i}] Since $s_0+\cdots+s_{k-1} =
      \binom{k}{2}$ and $s_0+\cdots +s_{k} \geq \binom{k+1}{2}$ we have
      $s_k\geq \binom{k+1}{2} - \binom{k}{2} = k$ and hence
      $v_0=s_k-k\geq 0$. Since $s$ is weakly increasing it is clear that
      $v$ is weakly increasing as well. Moreover, the length of $v$ is
      $n-k$ and $v_{n-k-1}=s_{n-1}-k \leq n-1-k$.
    \item[\eqref{landaus-thm-ii}] For $1 \leq j < n-k$ we have
      \begin{align*}
        v_0 + v_1 + \cdots + v_{j-1}
        &= s_k + s_{k+1} + \cdots s_{k+j-1} - jk \\
        &= s_0 + \cdots + s_{k+j-1} - (s_0 + \cdots + s_{k-1}) - jk \\
        &\geq \binom{k+j}{2}-\binom{k}{2}-jk = \binom{j}{2}.
      \end{align*}
    \item[\eqref{landaus-thm-iii}] Similarly,
      $v_0 + v_1 + \cdots + v_{n-k-1} = \binom{n}{2}-\binom{k}{2}-(n-k)k
      = \binom{n-k}{2}$.
  \end{enumerate}
  This concludes the proof.
\end{proof}

A direct consequence of Lemma~\ref{direct-sum-lemma}, above, is that every score sequence
$s$ can be uniquely written as a direct sum
$s=t_1\oplus t_2\oplus\cdots\oplus t_k$ of nonempty strong score sequences;
in this context, the $t_i$ will be called the \emph{strong summands} of
$s$. In terms of underlying tournaments we have the picture:
\[
  \includegraphics[height=3.4em]{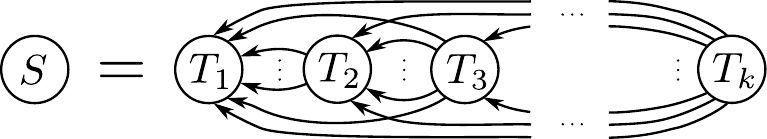}
\]

We are now almost in a position to define the promised map $\Phi$, but
first a couple of definitions. Assume that we are given a score sequence
$s=(s_0,s_1,\dots,s_{n-1})\in S_n$.
\begin{itemize}
\item For any integer $j$, let $s+j$ denote the sequence
  obtained by adding $j$ to each element of $s$, reducing modulo $n$,
  and sorting the outcome in nondecreasing order.  Note that $s+j$ need
  not be a score sequence even though $s$ is. E.g.\ $s=(1,1,1)$ is a
  score sequence, but $s+1=(2,2,2)$ is not. On the other hand, if
  $s=(0,1,2)$ then $s+1=s$ is a score sequence. A characterization of
  when $s+j$ is a score sequence will be given in Lemma~\ref{prelude}.

\item Let $\mu(s+j)$ denote the multiset $\{s_0+j,s_1+j,\dots,s_{n-1}+j\}$ with
  elements in the cyclic group $\ZZ_n$.
\end{itemize}

Given a pointed score sequence $(s,i)\in \ptS_n$, write
$s = t_1\oplus t_2\oplus\cdots\oplus t_k$ and
let $j$ be the smallest index such that $i < |t_1\oplus\cdots\oplus t_j|$.
Another way to define $j$ is as the smallest prefix $t_1\oplus \cdots \oplus t_j$
of strong summands of $s$ that begins $s_0,s_1,\ldots,s_{i}$.
Define the two score sequences $u$ and $v$ by
\[
  u = t_1\oplus\cdots\oplus t_j \quad\text{and}\quad
  v = t_{j+1}\oplus\cdots\oplus t_k.
\]
Finally, we let
\begin{equation*}
  \Phi(s, i) = \bigl(\mu(u-i), v\bigr).
\end{equation*}
As an example, consider the score sequence
$s = (0, 2, 2, 3, 3, 5, 7, 7, 7)$; its decomposition into strong
summands is $s = (0)\oplus (1,1,2,2)\oplus (0)\oplus (1,1,1)$. With
$i=3$ we get $u=(0)\oplus (1,1,2,2) = (0,2,2,3,3)$,
$v=(0)\oplus (1,1,1)=(0,2,2,2)$, $u-3 = (0,0,2,4,4)$ and so
$\Phi(s, 3)= \bigl(\{0,0,2,4,4\}, (0,2,2,2)\bigr)$.

\section{Proof of the main result}

Our aim is to prove Theorem~\ref{main-result}, but first we need to
establish a number of lemmas. Let $T_n$ be the set of strong score
sequences of length $n$.

\begin{lemma}\label{lemmasix}
  For any strong score sequence $s\in T_n$, the $n$ multisets
  \[\mu(s+j) = \{s_0 + j, s_1 + j, \dots, s_{n-1} + j\},\quad\text{for $j\in [0,n-1]$},
  \]
  are all distinct.
\end{lemma}

\begin{proof}
  Assume $j\in [0,n-1]$ is such that $\mu(s+j)$ and
  $\mu(s)$ are equal as multisets over $\ZZ_n$. Note that there is no
  loss of generality here: assuming that $\mu(s+j_1)=\mu(s+j_2)$
  is equivalent to assuming that $\mu(s+j)=\mu(s)$ with $j=j_2 - j_1$.
  Let
  us write the values $s_0 + j, \dots, s_{n-1} + j$ in nondecreasing order
  after reducing modulo $n$. Since $s_0 \leq \cdots \leq s_{n-1}$ the
  result must be a cyclic shift of the original order, say, $s_k +
  j,\dots,s_{n-1}+j, s_0+j, \dots, s_{k-1}+j$ for some index $k$. As
  each $s_i$ is less than $n$ and $j<n$ we find that those elements must
  equal
  \begin{equation}\label{sorted-list}
    s_k + j-n,\,\dots,\,s_{n-1}+j-n,\, s_0+j,\, \dots,\, s_{k-1}+j.
  \end{equation}
  Since we are assuming that $\mu(s+j)$ and $\mu(s)$
  are equal as multisets over $\ZZ_n$, the sum
  of the elements listed in \eqref{sorted-list} must equal the
  sum of the elements $s_0+s_1+\cdots+s_{n-1}=\binom{n}{2}$.  This
  implies $nj-(n-k)n=0$, which gives $j=n-k$.  Thus, the sequence
  \eqref{sorted-list} becomes
  \[s_k-k,\,\dots,\,s_{n-1}-k,\, s_0+n-k,\, \dots,\, s_{k-1}+n-k.
  \]
  By assumption we have $s_0=s_k-k$, $s_1=s_{k+1}-k$, and so on;
  thus
  \begin{align*}
    s_0 + \cdots + s_{n-k-1}
    &= s_k-k + \cdots + s_{n-1}-k \\
    &= \binom{n}{2} - (s_0 + \cdots + s_{k-1}) - (n-k)k \\
    &= \binom{k}{2} + \binom{n-k}{2} - (s_0 + \cdots + s_{k-1}).
  \end{align*}
  Suppose $k\geq 1$. Since our score sequence $s$ is strong we have $s_0
  + \cdots + s_{k-1} > \binom{k}{2}$, but then $s_0 + \cdots + s_{n-k-1}
  < \binom{n-k}{2}$ contradicting that $s$ is a score sequence. The only
  remaining possibility is that $k=0$ and $j=0$, which concludes the
  proof.
\end{proof}

Consider the following example concerning the strong summands of a score
sequence $s+j$. Let $s = (1, 1, 2, 2, 4, 6, 7, 7, 7, 8)$
and $j = 5$. We add $5$ to each element of $s$ to get the list of numbers $6$, $6$,
$7$, $7$, $9$, $11$, $12$, $12$, $12$, $13$, which when reduced modulo
$|s|=10$ reads $6$, $6$, $7$, $7$, $9$, $1$, $2$, $2$, $2$, $3$; finally
we sort this list to arrive at $s+5=(1, 2, 2, 2, 3, 6, 6, 7, 7, 9)$.
Note that $s$ and $s+5$ share the same strong summands only arranged
differently:
\begin{align*}
  s &= (1, 1, 2, 2)\oplus (0)\oplus (1, 2, 2, 2, 3);\\
  s+5 &= (1, 2, 2, 2, 3)\oplus (1, 1, 2, 2)\oplus (0).
\end{align*}
As is detailed in the following lemma, this is not a
coincidence. See also Fig.~\ref{f:flip} where we show how this may be
interpreted in terms of tournaments.
\begin{figure}[ht]
  \centering
  \includegraphics[height=4.8em]{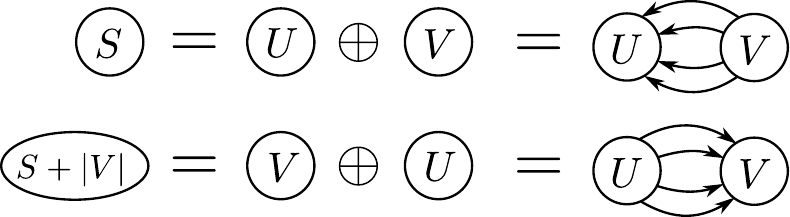}
  \caption{If $s=u\oplus v$, one may view $s+|v|$ as the score sequence
    of a tournament obtained by flipping the direction of all arrows
    pointing out of $V$ towards $U$. Here, $S,U,V$ are tournaments with
    score sequences $s,u,v$ respectively.} \label{f:flip}
\end{figure}    

\begin{lemma}\label{prelude}
  Let $s$ be any nonempty score sequence of length $n$ and let
  $j\in [1,n-1]$.  Then $s+j$ is a score sequence if and only if there
  are score sequences $u$ and $v$ with $|v| = j$, such that
  $s = u\oplus v$. In that case, $s + j = v \oplus u$.
\end{lemma}
\begin{proof}
  Let $s = (s_0,s_1,\ldots,s_{n-1})$. Define the sequences
  \begin{align*}
    u &= (s_0,s_1,\ldots,s_{n-j-1}); \\
    v &= (s_{n-j} + j - n, s_{n-j+1} + j - n, \ldots, s_{n-1} + j -n)
  \end{align*}
  of length $n-j$ and $j$, respectively. Assume that $u$ and $v$ are
  score sequences. Clearly, $s = u \oplus v$. Since the length of $v$ is
  $j$ we find that
  \begin{align*}
    v\oplus u = (s_{n-j} + j - n, \ldots, s_{n-1} + j
    -n,s_0+j,s_1+j,\ldots,s_{n-j-1}+j).
  \end{align*}
  If we consider the elements of this sequence
  modulo $n$ we may add $n$ to the first $j$ of them. If we then
  sort the elements in nondecreasing order we obtain the sequence
  $(s_0+j,\ldots,s_{n-1}+j)$. Thus, $s+j = v\oplus u$,
  which is a score sequence. For the other direction, assume that $s+j$
  is a score sequence.  Then by definition of $s+j$ there is an $\ell
  \in [0,n-1]$ such that $s+j$ is
  \begin{equation*}
     (s_\ell + j-n,s_{\ell+1}+j-n,\ldots,s_{n-1}+j-n,s_0+j,s_1+j,\ldots,s_{\ell-1}+j).
  \end{equation*}
  Since $s+j$ is a score sequence of length $n$, we find by item~\eqref{landaus-thm-iii} in
  Landau's theorem that
  \begin{align*}
    \binom{n}{2}
    &= \sum_{i=\ell}^{n-1}(s_i+j-n)+\sum_{i=0}^{\ell-1}(s_i+j) \\
    &= \sum_{i=0}^{n-1} s_i + (n-\ell)(j-n) + \ell j
    \,= \binom{n}{2} + n(\ell - n + j).
  \end{align*}
  Thus, $\ell = n-j$. Now consider the sum of the first $j$ terms
  in $s+j$; let us call this sum $K$. By Landau's theorem we have
  \begin{align*}
    \binom{j}{2} \leq K
    &= \sum_{i=n-j}^{n-1} (s_i+j-n)\\
    &= \sum_{i=n-j}^{n-1} s_i + j(j-n) \\
    &= \binom{n}{2} - \sum_{i=0}^{n-j-1}s_i + j(j-n) \\
    & \leq \binom{n}{2} - \binom{n-j}{2}+ j(j-n) \,=\, \binom{j}{2}.
  \end{align*}
  In particular, the two inequalities above are in fact equalities and
  $K = \binom{j}{2}$.  From Lemma \ref{direct-sum-lemma} we deduce that
  $u$ and $v$ as previously defined are score sequences and that
  $s+j = v\oplus u$. As before, it is clear that $s = u\oplus v$.
\end{proof}

The previous lemma may be equivalently stated as follows:

\begin{lemma}\label{lemmaeight}
  Let $s=t_1\oplus t_2\oplus\cdots\oplus t_k$ be any nonempty score
  sequence decomposed into its strong summands, and let $j\in [1,n-1]$.
  Then $s+j$ is a score sequence if and only if there is an $\ell$ such
  that
  $$j = |t_\ell| + |t_{\ell+1}| + \cdots + |t_k|
  $$
  in which case
  $s+j = t_\ell \oplus t_{\ell+1} \oplus \cdots \oplus t_k \oplus t_1 \oplus \cdots \oplus t_{\ell-1}$.
\end{lemma}

Our next lemma will be essential for proving that the mapping $f$ defined in
Lemma~\ref{last-strong-component}, below, is surjective.

\begin{lemma}\label{surjectivity-helper}
  For all multisets $M\in\EGZ_n$ there is a score sequence $s\in S_n$
  and a constant $j\in[0,n-1]$ such that $\mu(s+j)=M$.
\end{lemma}

\begin{proof}
  Let $M\in\EGZ_n$. By definition of $\EGZ_n$, the sum of the members of
  $M$ is $\binom{n}{2}$ modulo $n$.  We start by showing that for some
  integer $j$ the members of the multiset $M + j = \{x+j: x\in M\}$ will
  have sum exactly $\binom{n}{2}$ (without reducing the sum modulo
  $n$). Suppose that the sum of the members of $M$ is greater than
  $\binom{n}{2}$. We will show that we can always choose a $j$ such that
  $M + j$ has a smaller sum. For $k\in[0, n - 1]$, let $y_k$ be the
  number of elements $x\in M$ such that $x\leq k$.  Suppose that
  $y_k > k$ for all $k$.  Then $y_0 \geq 1$ and there is at least one
  zero in $M$.  Similarly, $y_1 \geq 2$ so in addition to that zero
  there is a value equal to at most one.  Continuing like this we bound
  our values from above by $0, 1, \dots, n - 1$.  The sum of these
  values is, however, at most $\binom{n}{2}$. Since we assumed that the
  sum is greater than $\binom{n}{2}$ we conclude that the assertion that
  $y_k > k$, for all $k$, is false. For the remainder of the argument,
  let $k$ be the smallest index such that $y_k \leq k$.

  If $k=0$, then there are no zeroes in $M$, so subtracting one from all
  the elements in $M$ simply causes the sum to drop by $n$ since no
  modulo reductions occur. Thus, we can assume $k > 0$.  Then
  $k - 1 < y_{k - 1} \leq y_k \leq k$ and hence $y_k=k$.  We now
  subtract $k + 1$ from all the elements in $M$. Exactly $y_k$ of these
  values will be below zero.  So we decrease the sum by $n(k + 1)$ but
  reducing modulo $n$ adds $kn$ to the sum.  The combined effect is to
  decrease the sum by $n$.  In this way we can always decrease the sum
  down to $\binom{n}{2}$. The proof that the sum be can be increased in
  the same way is nearly identical.

  Let $M=\{x_1,\ldots,x_n\} \in \EGZ_n$ with $x_1 \leq \dots \leq x_n$.
  In light of the last two paragraphs, we can assume that
  $x_1+\cdots+x_n=\binom{n}{2}$.  The sequence $(x_1,\ldots,x_n)$
  satisfies items $\eqref{landaus-thm-i}$ and $\eqref{landaus-thm-iii}$
  of Landau's theorem. To prove that $(x_1,\ldots,x_n)$ is a score
  sequence it remains to prove item $\eqref{landaus-thm-ii}$,
  namely that $x_1 + \dots + x_k \geq \binom{k}{2}$ for $k\in [1,n-1]$.
  Suppose that
  $x_1 + \dots + x_k < \binom{k}{2}$ for some $k$ in $[1,n-1]$.  For
  now, let $k$ be the largest such $k$. Since $k$ is maximal we have
  $x_1 + \dots + x_{k+1} \geq \binom{k+1}{2}$, which gives us
  $x_{k+1} \geq \binom{k+1}{2} - x_1 - \dots - x_k > \binom{k+1}{2} -
  \binom{k}{2} = k$.  If $x_k \geq k$ then we can decrement $k$
  until this no longer holds true; note that
  $x_1 + \dots + x_{k-1} < \binom{k - 1}{2}$ will continue to hold while
  $x_k \geq k$ and $x_1 + \dots + x_k < \binom{k}{2}$.  We eventually arrive
  at a $k$ such that $x_1 + \dots + x_k < \binom{k}{2}$, $x_k < k$ and
  $x_{k+1} \geq k$.
  The values $x_1, \dots, x_k$ are thus precisely the
  values among $x_1, \dots, x_n$ that are smaller than $k$.
  Define $y_i$ as before and let $w_i = y_i - i - 1$.
  The number of values equal to $r$ among $x_1, \dots, x_k$ is
  $y_r - y_{r-1}$ as long as $r < k$, and so
  \[
    x_1 + \dots + x_k
    = 0\cdot y_0 + 1\cdot (y_1-y_0) + \dots + (k-1)(y_{k-1} - y_{k-2}).
  \]
  This implies
  $(k-1)y_{k-1} - y_{k-2} - \dots - y_0 = x_1 + \dots + x_k <
  \binom{k}{2}$, but we know that $y_{k-1} = k$, so we must have
  $k^2 - \binom{k}{2} < y_{k-1} + \dots + y_0$.  Rewriting in terms of
  $w_k$ we get
  $\binom{k + 1}{2} < w_0 + \dots + w_{k-1} + 1 + 2 + \dots + k$, which
  is equivalent to $w_0 + \dots + w_{k-1} > 0$.  Thus, if we can choose
  a shift (i.e.\ choose a constant to add to the elements of $M$) such
  that the prefix sums of the $w_i$ are non-positive, then the original
  set must be a score sequence.

  Consider the values $k$ such that $w_k = 0$ or, equivalently, $y_k =
  k + 1$.  If we subtract $k + 1$ from every value in our multiset, then exactly
  $k + 1$ of them will become negative.  Thus, our sum decreases by $n(k
  + 1)$ but the modulo reduction adds $n(k + 1)$ back to the sum, so it
  remains unchanged.  Let us consider what this does to our sequences
  $(y_i)$ and $(w_i)$.  Denote the new sequences after the subtraction by
  $(y_i')$ and $(w_i')$.  Consider the value of $y_i'$. It counts the number
  of values from $0$ to $i$ after the subtraction, which are values from $k +
  1$ to $k + 1 + i$ modulo $n$ before the subtraction.
  We consider two cases. 

  (a) If $k + 1 + i < n$, then $y_i' = y_{k+1+i} - y_k$.
  Since $y_k = k + 1$ this gives us in terms of $w_i$ that $w_i' + i + 1
  = w_{k+i+1} + k + i + 1 + 1 - k - 1$, which simplifies to $w_i' =
  w_{k+i+1}$.

  (b) If $k + 1 + i \geq n$, then
  $y_i' = y_{n - 1} - y_k + y_{k + 1 + i - n}$. We already know that $y_k = k + 1$ and
  $y_{n - 1} = n$.  In terms of the $w_i$ we have
  $w_i' + i + 1 = n - k - 1 + w_{k+1+i-n}+k+1+i-n+1$, which reduces to $w_i' = w_{k+1+i-n}$.

  Hence, in both cases we have $w_i' = w_{k+i+1}$
  when indices are considered modulo $n$.  We can thus cyclically
  permute the sequence $(w_i)$ as long as the final value continues to
  be $0$ and still have it correspond to a shift of our original multiset
  $M$ such that the shift has sum $\binom{n}{2}$.

  By the same argument as above (when $k=n$) we have
  \[
    x_1 + \dots + x_n = (n - 1)y_{n - 1} - y_{n - 1} - y_{n - 2} - \dots - y_0.
  \]
  Furthermore, $y_{n-1} = n$ and $x_1 + \dots + x_n = \binom{n}{2}$. Thus,
  $y_0 + \dots + y_{n-1} = \binom{n}{2}$, which gives us
  $w_0 + \dots + w_{n-1} = 0$.  Hence, we can consider the sequence of sums
  of values between the
  zeroes that appear in $(w_i)$. This new sequence $(g_i)$ also has sum
  zero and we can cyclically permute the sequence $(w_i)$ in the way we
  want if and only if we can permute the sequence $(g_i)$ to satisfy the same
  desired property.  This is because either all values between two
  zeroes in $(w_i)$ are of the same sign, or the positive values all
  come after the negative ones.  The reason for this is that one cannot
  go from positive values to negative without intercepting zero since we
  decrease by at most one at a time.

  A sequence of integers with sum zero can be cyclically permuted to
  produce a sequence with non-positive prefix sums.  To see this let
  $z_1,\ldots,z_n$ be integers that sum to zero.  By applying Raney’s
  lemma~\cite[\S 7.5, p.\ 345]{Concrete} to the sequence
  $-(z_1-1),-z_2,\ldots,-z_n$ there exists a cyclic permutation of this
  sequence having all positive prefix sums. Negate this new sequence and
  you have one with all negative prefix sums. Then add 1 to the element
  where $z_1 - 1$ ended up. This will increase any given prefix sum by
  at most 1, taking them from negative to at most non-positive, which
  proves the non-positive prefix sums claim.
  This allows us to conclude that the prefix sums, as discussed above,
  are non-positive and so the original set must be a score sequence,
  which completes our proof.
\end{proof}

Let $T(x)$ be the generating function for the number of strong score
sequences according to length. Note that any nonempty score sequence $s$
can be written $s=u\oplus t$, where $u$ is a score sequence and $t$ is
the last strong summand of $s$, and thus $S(x) = 1 + S(x)T(x)$, or,
equivalently, $S(x)=(1-T(x))^{-1}$.  Let
\[\ptT_n = T_n \times [0,n-1]
\]
be the set of \emph{pointed strong score sequences} of length $n$, and let
\[
  \ptT_n(t) = \sum_{(s,i)\in \ptT_n} t^{|s^{-1}(i)|}\quad\text{and}\quad
  \ptT(x,t) = \sum_{n\geq 1} \ptT_n(t) x^n.
\]

\begin{lemma}\label{last-strong-component}
  Assume $n\geq 1$. Select those pointed score sequences $(s,i)$ in
  $\ptS_n$ whose distinguished position is in the last strong summand,
  and denote the resulting set $L_n$. That is,
  \[L_n=\bigl\{(s,i)\in\ptS_n: s=t_1\oplus \cdots \oplus t_k \mbox{ and } i\in [n-|t_k|,n-1]\bigr\}.
  \]
  Then the mapping $f:L_n\to \EGZ_n$ defined by
  \[ f(s,i) = \mu(s-i)
  \]
  is a bijection. Moreover, if $M = f(s,i)$, then
  $|M|_0 = |s^{-1}(i)|$.  In terms of generating functions we have
  \[ \EGZ(x,t) = S(x)\ptT(x,t).
  \]
\end{lemma}

\begin{proof}
  We start by proving that $f$ is injective. To that end, assume that
  $(u,i)$ and $(v,j)$ in $L_n$ are such that $f(u,i) = f(v,j)$; that is,
  \[\mu(u-i) = \mu(v-j) = M,\quad  \text{where $M\in \EGZ_n$}.
  \]
  If $u=v$, then $i=j$ by Lemma~\ref{lemmasix}, so we may assume that
  $u$ and $v$ are different score sequences.  Without loss of generality
  we may further assume that $i\leq j$. Now, $\mu(u-i) = \mu(v-j)$ holds
  by assumption, and this multiset identity is true if and only if
  $\mu(u) = \mu(v+(j-i))$.  Thus, $u=v+m$ with $m=j-i$.  Assume that the
  decomposition of $v$ into its strong summands is
  $v = t_1 \oplus \cdots \oplus t_k$. By
  Lemma~\ref{lemmaeight}, the score sequences $u$ and $v$ contain the
  same strong summands up to a cyclic permutation.  To be more precise,
  there is an $\ell\geq 1$ such that
  \[
    m=|t_{\ell}|+\cdots+|t_k|\quad\text{and}\quad
    u = t_{\ell+1} \oplus \cdots\oplus t_k \oplus t_1 \oplus \cdots \oplus t_{\ell}.
  \]
  By definition of the set $L_n$ we have $i \in [n-|t_{\ell}|,n-1]$ and
  $j \in [n-|t_{k}|,n-1]$.  Therefore
  $m=j-i \leq n-1-(n-|t_{\ell}|) = |t_{\ell}|-1$.  This would however imply
  $|t_{\ell}|+\cdots+|t_k|< |t_{\ell}|$, which is impossible.
  Therefore no such $i$ and $j$ can exist and the mapping $f$ is
  injective.

  Next we prove surjectivity.  We want to prove that every $M\in\EGZ_n$
  is the image of some score sequence $s \in S_n$ with a distinguished
  element $i$ in its last strong summand, i.e.\ $(s,i)\in L_n$. By
  Lemma~\ref{surjectivity-helper} every $M \in \EGZ_n$ is the image of some such $(s, j)$
  where we place no restriction on $j$. If $j$ is smaller
  than the size of the last strong summand of $s$, then $f$
  maps the pointed score sequence $(s,n-j)$ to our multiset $M$.
  If $j$ is greater than the size of the last strong summand of $s$,
  then we can move the last summand to the front and decrease $j$ by
  the corresponding size by Lemma~\ref{lemmaeight}. In this way
  we will eventually end up in the first case. Thus, we end up with a
  score sequence that maps to our multiset $M$.

  We have shown that $f$ is a bijection. Assume that $(s,i) \in L_n$ and
  $M=f(s,i)$. By definition of $f$ it is clear that the number of zeros
  in $M$ corresponds to the number of $j$ for which $s_j=i$, and hence
  $|M|_0 = |s^{-1}(i)|$.  Every $(s,i) \in L_n$ can be uniquely
  identified with a pair $(s',(t,j))$, where $s=s' \oplus t$, $t$ is the
  last strong summand of $s$, and $j=i-|s'|$.  In other words, $s'$ is
  the score sequence $s$ with its last strong summand $t$ removed, and
  $j$ is the distinguished position relative to $t$ rather than
  $s$. Assuming $k=|s'|$, we have $s' \in S_k$ and
  $(t,j) \in \ptT_{n-k}$ with $|t^{-1}(j)|=|M|_0$.  Consequently,
  $\EGZ(x,t) = S(x)\ptT(x,t)$ as claimed.
\end{proof}

We are now finally in a position to prove our main result.

\begin{proof}[Proof of Theorem~\ref{main-result}]
  We will give a combinatorial proof of the power series identity
  $\ptS(x,t) = \EGZ(x,t)S(x)$ by showing that the mapping
  \[
  \Phi: \ptS_n \to \bigcup_{k=1}^n \EGZ_k\times S_{n-k}
  \]
  defined in Section~\ref{section-main-thm} is a bijection.  Recall that
  $\Phi(s,i)=(\mu(u-i),v)$ is defined by writing $s$ in terms of its
  strong summands, say $s=t_1\oplus t_2\oplus\cdots\oplus t_k$, and then
  splitting the score sequence $s=u \oplus v$, where
  $u=t_1\oplus\cdots\oplus t_j$ and $j$ is the smallest index such that
  $i < |t_1\oplus\cdots\oplus t_j|$, and $v$ is the remainder. Let $L_n$
  be the set defined in Lemma~\ref{last-strong-component}. Note that
  $(u,i)\in L_{|u|}$. In fact, it is easy to see that
  $(s,i)\mapsto \bigl((u,i), v\bigr)$ is a bijection from $\ptS_n$ to
  $\cup_{k=1}^n L_k\times S_{n-k}$. Moreover, by
  Lemma~\ref{last-strong-component}, $(u,i)\mapsto\mu(u-i)$ is a
  bijection from $L_k$ to $\EGZ_k$, and hence $\Phi$ is a bijection as
  well. Finally, $(u,i)\mapsto\mu(u-i)$ has the property
  that $|\mu(u-i)|_0 = |u^{-1}(i)|$ and so $\ptS(x,t) = \EGZ(x,t)S(x)$,
  as claimed.
\end{proof}

\begin{corollary}\label{catalan-conundrum}
  We have
  \[\vptS(x) = xC(x)S(x),
  \]
  where $C(x)$ is the generating function for the Catalan numbers.
\end{corollary}

\begin{proof}
  From Theorem~\ref{main-result} we get
  \[\vptS(x) = \ptS(x,1)-\ptS(x,0) = \bigl(\EGZ(x,1)-\EGZ(x,0)\bigr)S(x).
  \]
  It thus suffices to show that $\EGZ(x,1)-\EGZ(x,0)=xC(x)$. Note that
  the coefficient of $x^n$ in $\EGZ(x,1)-\EGZ(x,0)$ is the number of
  multisets of $\EGZ_n$ that contain at least one zero. On removing a
  single zero from such a multiset we are left with a multiset of size
  $n-1$ whose sum is $\binom{n}{2}$ modulo $n$. There are $C_{n-1}$ such
  multisets~\cite[ Exercise 6.19jjj]{EC2}.
\end{proof}

\begin{corollary}\label{exp-formula}
  We have
  \[ S(x) \,=\, \exp\Biggl(\,\sum_{n\geq 1} \frac{|\EGZ_n|}{n} x^n\Biggr).
  \]
\end{corollary}
\begin{proof}
  Since $\ptS(x,1)=xS'(x)$, by letting $t=1$ in the main theorem
  we have $xS'(x)=\EGZ(x,1)S(x)$, or, equivalently,
  $x\bigl(\log S(x)\bigr)'=\EGZ(x,1)$. Thus
  \[ \log S(x) = \int_0^x\EGZ(u,1)\frac{1}{u} \,du = \sum_{n\geq 1} |\EGZ_n|\frac{\,x^n}{n},
  \]
  as claimed.
\end{proof}

We end this section by comparing our result
(Corollary~\ref{exp-formula}) with earlier results on the enumeration of
the \emph{ordered} score sequences $(s_0,s_1,\dots,s_{n-1})$, also
called \emph{score vectors}. That is, if $G$ is a tournament on the
vertex set $\{v_0, v_1, \dots, v_{n-1}\}$, then $s_i$ is the out-degree
of $v_i$ in $G$. For instance, while there are only two score sequences
of length $3$, namely $(1,1,1)$ and $(1,2,3)$, there are 7 score vectors
of length $3$: the vector $(1,1,1)$ together with the 6 permutations of
$(1,2,3)$.

Stanley and Zaslavsky~\cite{Stanley1980} have shown that the number of
score vectors of length $n$ equals the number of (labeled) forests on
$n$ nodes. A combinatorial proof was subsequently given by Kleitman and
Winston~\cite{Kleitman1981}. Cayley~\cite{Cayley2009} famously gave the
formula $n^{n-2}$ for the number of trees on $n$ nodes. From the theory
of exponential generating functions it immediately follows that
\[
  \exp\Biggl(\,\sum_{n\geq 1}n^{n-2}\frac{x^n}{n!}\,\Biggr)
\]
is the exponential generating function of forests, and thus also of score
vectors.

\section{The number of score sequences}

If two power series $A(x)=1+\sum_{n\geq 1}a_nx^n$ and
$B(x)=\sum_{n\geq 1}b_nx^n$ satisfy $xA'(x)/A(x) = B(x)$ and hence
$\log A(x) = \sum_{n\geq 1}b_nx^n/n$, then one readily obtains a closed
formula for $a_n$ by expanding and identifying coefficients in
$A(x) = \exp\bigl(b_1x^1/1\bigr)\exp\bigl(b_2x^2/2\bigr)\cdots$.
Applying this to the equation in Corollary~\ref{exp-formula} we arrive at
\begin{equation}
  |S_n|
  = \frac{1}{n!}\sum_{\pi\in \Sym(n)}\prod_{\ell\in C(\pi)}|\EGZ_{\ell}|,
\end{equation}
where $\Sym(n)$ is the symmetric group of degree $n$ and $C(\pi)$ encodes
the cycle type of $\pi$; i.e.\ there is an $\ell\in C(\pi)$ for each
$\ell$-cycle of $\pi$. While having the virtue of being closed, this
formula does not lend itself to quickly calculating $|S_n|$. For that
purpose the following recursion is better suited.

\begin{corollary}\label{recursion}
  For $n\geq 1$,
  \begin{align*}
    |S_n|
    &= \frac{1}{n}\sum_{k=1}^{n}|S_{n-k}||\EGZ_k| \\
    &= \frac{1}{n}\sum_{k=1}^{n}|S_{n-k}|
      \frac{1}{2k}\sum_{d | k}(-1)^{k-d}\varphi(k/d)\binom{2d}{d}.
  \end{align*}
\end{corollary}
\begin{proof}
  Since $\ptS_n=S_n\times [0,n-1]$ there is a simple relation
  $|\ptS_n|=n|S_n|$ between the number of pointed score sequences and
  the number of (plain) score sequences. Thus the result immediately
  follows from identifying coefficients in the main theorem, with
  $t=1$, and using Equation~\ref{EGZ-cardinality} in
  Section~\ref{section-intro}.
\end{proof}

This allows us
to calculate all values of $|S_k|$ for $k \leq n$ in 
$\Theta(n^2)$
time, assuming constant time integer operations. This is an improvement
on earlier results by Narayana and Bent~\cite{NarayanaBent1964}. Their 
recursive formula can be implemented to find $|S_n|$ in 
$\Theta(n^3)$
time,
but no faster since their recursive function must always visit 
$\Theta(n^3)$
states to do so; to get all $|S_k|$ for $k \leq n$ takes 
$\Theta(n^4)$
time
due to lack of overlap in the states recursively visited for different $k$.

Since $S(x) = (1 - T(x))^{-1}$ 
this recursive computation method can be extended to
$|T_k|$. We can
rewrite the equation as 
$1 + T(x) = S(x) - (S(x) - 1)T(x)$
which gives
us the recursion
\[
  |T_n| = |S_n| - \sum_{i = 1}^{n - 1} |T_i| |S_{n-i}|.
\]
We first calculate the values $|S_k|$ and use this recursion
to calculate all the values $|T_k|$ for $k \leq n$ in
$\Theta(n^2)$
time. This
is the same method as used by Stockmeyer~\cite{Stockmeyer2022}, just
calculating the underlying $|S_k|$ faster which brings the total time complexity
down from
$\Theta(n^4)$
to
$\Theta(n^2)$.

\bibliographystyle{plain}
\bibliography{references}

\end{document}